\documentclass[12pt]{article}

\usepackage{latexsym, amssymb,amsmath,amscd, amsfonts, amsthm, ifpdf, fancybox,float,  geometry, epsfig, graphicx, color, colordvi, setspace,caption,subcaption,indentfirst,soul}
\usepackage{ytableau}
\usepackage{genyoungtabtikz}

\allowdisplaybreaks[4]
\newtheorem{thm}{Theorem}[section]

\newtheorem{conj}[thm]{Conjecture}

\newtheorem{lem}[thm]{Lemma}
\newtheorem{core}[thm]{Corollary}

\numberwithin{equation}{section}

\sethlcolor{yellow}
\begin{document}
	
\begin{center}

 {\Large \bf On the hook length biases of the $2$- and $3$-regular partitions}
\end{center}

\begin{center}
{  Wenxia Qu}$^{1}$,  and
  {Wenston J.T. Zang}$^{2}$ \vskip 2mm

   $^{1,2}$School of Mathematics and Statistics, Northwestern Polytechnical University, Xi'an 710072, P.R. China\\[6pt]
   $^{1,2}$ MOE Key Laboratory for Complexity Science in Aerospace, Northwestern Polytechnical University, Xi'an 710072, P.R. China\\[6pt]
$^{1,2}$ Xi'an-Budapest Joint Research Center for Combinatorics, Northwestern Polytechnical University, Xi'an 710072, P.R. China\\[6pt]
   \vskip 2mm

    $^1$quwenxia0710@mail.nwpu.edu.cn, $^2$zang@nwpu.edu.cn
\end{center}

\vskip 6mm \noindent {\bf Abstract.} Let $b_{t,i}(n)$ denote the total number of the $i$ hooks in the $t$-regular partitions of $n$. Singh and Barman (J. Number Theory {  264} (2024), 41--58) raised two conjectures on $b_{t,i}(n)$. The first  conjecture is on the positivity of  $b_{3,2}(n)-b_{3,1}(n)$ for $n\ge 28$. The second conjecture states that when $k\ge 3$, $b_{2,k}(n)\ge b_{2,k+1}(n)$ for all $n$ except for $n= k+1$. In this paper, we confirm the first conjecture. {Moreover, we show that for any odd $k\ge 3$, the second conjecture fails for infinitely many $n$.} {Furthermore, we verify that the second conjecture holds for $k=4$ and $6$.} We also propose a conjecture on the even case $k$, which is a modification of  Singh and Barman's second conjecture.



\section{Introduction}\label{section-Introduction}

This paper  {focuses} on the  hook length biases in $2$- and $3$-regular partitions.
Recall that an integer partition  of $n$ is a finite sequence of non-increasing positive integers $\lambda=(\lambda_1,\ldots,\lambda_\ell)$ such that $\lambda_1+\cdots+\lambda_{\ell}=n$.  A $t$-regular partition is  a partition  {in which} all parts are not  divisible by $t$. A  $t$-distinct partition is a partition of $n$ where parts can only appear at most $t-1$ times.  We use $b_t(n)$(resp. $d_t(n)$) to denote the number of $t$-regular(resp. $t$-distinct) partitions of $n$. It is well known that $b_t(n)=d_t(n)$. For example, there are seven $3$-regular partitions of $6$ as given below,
\[(5,1),(4,2),(4,1,1),(2,2,2),(2,2,1,1),(2,1,1,1,1),(1,1,1,1,1,1).\]
Hence $b_3(6)=7$. The seven $3$-distinct partitions of $6$ are listed as follows,
\[(6),(5,1),(4,2),(4,1,1),(3,3),(3,2,1),(2,2,1,1).\]
Thus $d_3(6)=b_3(6)=7$. We also recall that the \textit{Young diagram} of a partition $(\lambda_1,\ldots,\lambda_\ell)$ is a left-justified array boxes in which $i$-th row has $\lambda_i$ boxes. For example, the Young diagram of partition $(4,3,3,2,1)$ is shown in Figure \ref{F1}. The hook length of a box in Young diagram is the   number of the boxes directly to its right or directly below it and including itself exactly once. For example, Figure \ref{F2} shows the hook length of each box in the Young diagram of partition $(4,3,3,2,1)$.
\begin{figure}[htbp!]
	\centering
	\begin{subfigure}[b]{0.48\linewidth}
		\begin{center}
		\begin{ytableau}[]
        \ & \ & \ & \ \\
        \ & \ & \ \\
        \ & \ & \ \\
        \ & \ \\
        \
        \end{ytableau}
		\caption{The Young diagram of partition $(4,3,3,2,1)$}
		\label{F1}
		\end{center}
	\end{subfigure}
	\begin{subfigure}[b]{0.48\linewidth}
		\begin{center}
		\begin{ytableau}[]
        8 & 6 & 4 & 1\\
        6 & 4 & 2 \\
        5 & 3 & 1 \\
        3 & 1 \\
        1
        \end{ytableau}
		\caption{The hook lengths of $(4,3,3,2,1)$}
		\label{F2}
		\end{center}
	\end{subfigure}
	\caption{The Young diagram of partition $(4,3,3,2,1)$ and its hook lengths}
	\label{fig:1}
\end{figure}

Hook length plays a crucial role in the theory of representation theory of the symmetric group $S_n$ and the general linear group $GL_n(C)$. Hook length also relates to the theory of symmetric functions as the number of the standard Young tableaux, which is well known as the hook-length formula, see \cite{Frame-Rhobinson-Thrall-1954,Young-1927} {for more details}. The Nekrasov-Okounkov formula (see \cite{Han-Nek-2010, Nek-Oko}) builds the connection between hook length and the theory of modular forms. There are many other studies on the hook length, including the asymptotic, combinatorial and arithmetic properties, especially the $t$-core partitions and the $t$-hook partitions, see \cite{Garvann-Kim, James-Kerber, Littlewood-Modular} for example.

This paper mainly {focuses} on the total number of certain hook length in all $t$-regular partitions. Let $b_{t,i}(n)$ denote the number of hooks of length $i$ in all $t$-regular partitions of $n$ and $d_{t,i}(n)$ denote the number of hooks of length $i$ in all  $t$-distinct partitions of $n$. Ballantine, Burson, Craig, Folsom and Wen \cite{ballantinehook2023}  proved that for any $n\ge 0$, $d_{2,1}(n)\ge b_{2,1}(n)$  and  $b_{2,i}(n)\ge d_{2,i}(n)$ for $i=2,3$. They  also proposed a conjecture that for every $i\ge 2$, there exists an integer $N_{i}$ such that $b_{2,i}(n)\ge d_{2,i}(n)$ for all $n\ge N_{i}$. This conjecture was proved by Craig, Dawsey and Han  \cite{craig-dawsey-han-inequalities}.    Singh and Barman \cite{singh2024hook} studied the biases among $b_{t,i}(n)$ and $d_{t,i}(n)$ for fixed $i$. To be specific, they proved that $b_{t+1, 1}(n)\ge b_{t,1}(n)$ for all $t\ge 2$, $n\ge 0$ and $d_{t+1,i}(n)\ge d_{t,i}(n)$ for any $t,i$, and $n\ge 0$. Moreover, they also showed that $b_{3,2}(n)\ge b_{2,2}(n)$ for all $n>3$ and $b_{3,3}(n)\ge b_{2,3}(n)$ for all $n\ge 0$. At the end of their paper, they conjectured that $b_{t+1,2}(n) \ge  b_{t,2}(n)$ for $t \ge 3$ and for all $n \ge 0$. The case $t=3$ of this conjecture was confirmed by Mahanta \cite{mahanta2024singh}.

{The hook length biases on the $2$- and $3$-regular partitions} were first studied by Singh and Barman in \cite{singh-barman-hook}. They proved that $b_{2,2}(n)\ge b_{2,1}(n)$ for all $n > 4$ and $b_{2,2}(n)\ge b_{2,3}(n)$ for all $n\ge0$. Moreover, they also proposed the following conjecture on the biases of $b_{3,2}(n)$ and $b_{3,1}(n)$.
\begin{conj}[\cite{singh-barman-hook}]\label{conj-1}
	For all $n\ge 28$, $b_{3,2}(n)\ge b_{3,1}(n)$.
\end{conj}

The first main result of this paper is to give a confirmation of the above conjecture.

\begin{thm}\label{thm-main1}
	Conjecture \ref{conj-1} is true.
	
\end{thm}

Singh and Barman also raised the following conjecture on $b_{2,k}(n)$ and $b_{2,k+1}(n)$.

\begin{conj}[\cite{singh-barman-hook}]\label{conj-2}
	For every integer $k\ge 3$, $b_{2,k}(n)\ge b_{2,k+1}(n)$ for all $n\ge 0$ and $n\ne k+1$.
\end{conj}

They checked this conjecture for $1\le n,k\le 10$. However, this conjecture is not valid for odd $k$. For example, when $k=3$ and $n=82$, we have $b_{2,3}(82)=515393$ and $b_{2,4}(82)=515487$, which is a {counterexample} to Conjecture \ref{conj-2}. In fact, we find that for any odd $k\ge 3$, Conjecture \ref{conj-2}
 does not hold.

\begin{thm}\label{thm-mainodd}
	For any $k\ge 3$  and $k$ is odd ,	there exists infinitely many $n$ such that
	\begin{equation}
		b_{2,k}(n)<b_{2,k+1}(n).
	\end{equation}
\end{thm}

Moreover,
we prove that Conjecture \ref{conj-2} is true when $k=4$ and $6$. This provides supporting evidence that Conjecture \ref{conj-2} maybe hold for even $k$.

\begin{thm}\label{thm-main21}
For $n\ge 0$ and $n\ne 5$ , we have
\begin{equation}\label{eq-mainb24ngeb25n}
	b_{2,4}(n)\ge b_{2,5}(n).
\end{equation}
\end{thm}
\begin{thm}\label{thm-main22}
For $n\ge 0$ and $n\ne 7$, we have
\begin{equation}\label{eq-mainb26ngeb27n}
	b_{2,6}(n)\ge b_{2,7}(n).		
\end{equation}
\end{thm}

After checking $8\le k\le 20$ and $n\le 10000$, we {propose} the following conjecture, which is a modification of Conjecture \ref{conj-2}.

\begin{conj}\label{conj-3}
For   even $k\ge 8$, $b_{2,k}(n)\ge b_{2,k+1}(n)$ for all $n\ge 0$ and $n\ne k+1$.
\end{conj}

This paper is organized as follows.    Section \ref{section-Proofthm1} is devoted to proving Theorem \ref{thm-main1}. To be specific, we first transform the generating function of $b_{3,2}(n)-b_{3,1}(n)$ into a difference of two $q$-series, and then combinatorially interpret each $q$-series as a set of partition triplets. We next divide these two sets into seven disjoint subsets. After building six bijections on the first six subsets and analyzing the cardinality of the seventh subset, we provide a semi-combinatorial proof of Theorem \ref{thm-main1}.   The proof of Theorem \ref{thm-mainodd} will be given in Section \ref{section-proofmainodd}. To this end,   we will show that there exists infinitely many negative coefficients of $q^n$ in $B_t(q)\left(\sum_{n=0}^{\infty}(b_{2,2t+1}(n)-b_{2,2t+2}(n))q^n\right)$, where $B_{t}(q)$ has only positive coefficients. In Section \ref{section-proofmain2}, we will prove Theorem \ref{thm-main21}  and Theorem \ref{thm-main22} by analyzing {the} coefficients of $q^n$ in the difference between the generating functions of $b_{2,i}(n)$ and $b_{2,i+1}(n)$ for $i=4$ and $6$.

\section{Proof of Theorem \ref{thm-main1}}\label{section-Proofthm1}

In this section, we give a proof of Theorem \ref{thm-main1}.  To this end, we first introduce some notations on integer partitions. Throughout this paper,
for a fixed partition $\lambda$ of $n$, let $\ell(\lambda)$ denote the length of $\lambda$ and let $|\lambda|$  denote the summation of $\lambda$. We also write a partition $\lambda$ of $n$ as $(1^{f_\lambda(1)},2^{f_\lambda(2)},\ldots,n^{f_\lambda(n)})$, where  $f_\lambda(t)$ denotes the number of appearances of $t$ in $\lambda$. For simplicity, we omit the term $k^{0}$ and write $k$ instead of $k^1$.    For example, let  $\lambda=(5,3,2,2,1,1,1)$ be a partition of $15$. Then $\ell(\lambda)=7$, $|\lambda|=15$ and $f_{\lambda}(5)=f_{\lambda}(3)=1$, $f_{\lambda}(2)=2$, $f_{\lambda}(1)=3$. We also write $\lambda$ as $(1^3,2^2,3,5)$.

Recall that Singh and Barman \cite{singh-barman-hook} gave the expressions on the generating functions of $b_{t,1}(n)$ and $b_{3,2}(n)$ as the following two theorems.

\begin{thm}[\cite{singh-barman-hook}]\label{thm-bt1n}
	For $t\ge 2$, we have
	\begin{equation}\label{eq-bt1n}
		\sum_{n=0}^{\infty}b_{t,1}(n)q^n=\frac{(q^t;q^t)_{\infty}}{(q;q)_{\infty}}\left(\frac{q}{1-q}-\frac{q^t}{1-q^t}\right).
	\end{equation}
\end{thm}

\begin{thm}[\cite{singh-barman-hook}]\label{thm-b32n}
	We have
	\begin{equation}\label{eq-b32n}
\sum_{n=0}^{\infty}b_{3,2}(n)q^n=\frac{(q^3;q^3)_{\infty}}{(q;q)_{\infty}}\left(\frac{q^2}{1-q}+\frac{q^2}{1-q^2}-\frac{2q^3}{1-q^3}\right).
	\end{equation}
\end{thm}

Here we use the standard $q$-series notation
\[(a;q)_n=\prod_{i=1}^{n}(1-aq^{i-1}),\quad
(a;q)_\infty=\prod_{i=1}^{\infty}(1-aq^{i-1}),\]
and
\[{m\choose n}_q=\frac{(q;q)_m}{(q;q)_n(q;q)_{m-n}}.\]

From the above two theorems,  we may transform the generating function of $b_{3,2}(n)-b_{3,1}(n)$ as in the following lemma.

\begin{lem}\label{thm1-lem-1}
The following identity holds:
	\begin{equation}\label{eq-thm1-lem-1}
		\sum_{n=0}^{\infty}(b_{3,2}(n)-b_{3,1}(n))q^n=-q(1+q^2+q^4)(1+q^3+q^6)\left(\frac{(q^{12};q^3)_{\infty}}{(1-q^2)(q^5;q)_{\infty}}-\frac{q^3(q^{12};q^3)_{\infty}}{(1-q^2)(q^4;q)_{\infty}}\right).
	\end{equation}
\end{lem}
\begin{proof}
	Letting $t=3$ in Theorem \ref{thm-bt1n}, we have
	\begin{equation}\label{eq-b31n}
		\sum_{n=0}^{\infty}b_{3,1}(n)q^n=\frac{(q^3;q^3)_{\infty}}{(q;q)_{\infty}}\left(\frac{q}{1-q}-\frac{q^3}{1-q^3}\right).
	\end{equation}
	Combining with Theorem \ref{thm-b32n}, we deduce that
	\begin{align}\label{eq-b32n-b31n}
		&\sum_{n=0}^{\infty}(b_{3,2}(n)-b_{3,1}(n))q^n\nonumber\\
		=&\frac{(q^3;q^3)_{\infty}}{(q;q)_{\infty}}\left(\frac{q^2}{1-q}+\frac{q^2}{1-q^2}-\frac{2q^3}{1-q^3}\right)-\frac{(q^3;q^3)_{\infty}}{(q;q)_{\infty}}\left(\frac{q}{1-q}-\frac{q^3}{1-q^3}\right)\nonumber\\
		=&\frac{(q^3;q^3)_{\infty}}{(q;q)_{\infty}}\left(-q+\frac{q^2}{1-q^2}-\frac{q^3}{1-q^3}\right)\nonumber\\
		=&\frac{(q^3;q^3)_{\infty}}{(q;q)_{\infty}}\left(\frac{-q(1-q-q^3+q^5)}{(1-q^2)(1-q^3)}\right).
\end{align}
Notice that
\begin{equation}\label{equ-temp-thm-main-1}
  1-q-q^3+q^5=(1-q)(1-q^4)-q^3(1-q).
\end{equation}
Multiplying
\[\frac{-q(q^3;q^3)_\infty}{(1-q^2)(1-q^3)(q;q)_\infty}\]
on both sides of \eqref{equ-temp-thm-main-1}, we have
		\begin{align}\label{eq-frac-q3-q3} &\frac{(q^3;q^3)_{\infty}}{(q;q)_{\infty}}\left(\frac{-q(1-q-q^3+q^5)}{(1-q^2)(1-q^3)}\right)\nonumber\\
=&-q\left(\frac{(q^3;q^3)_{\infty}}{(1-q^2)^2(1-q^3)^2(q^5;q)_{\infty}}-\frac{q^3(q^3;q^3)_\infty}{(1-q^2)^2(1-q^3)^2(q^4;q)_\infty}\right)\nonumber\\
=&-q\frac{(1-q^6)(1-q^9)}{(1-q^2)(1-q^3)}\left(\frac{(q^{12};q^3)_{\infty}}{(1-q^2)(q^5;q)_{\infty}}-\frac{q^3(q^{12};q^3)_{\infty}}{(1-q^2)(q^4;q)_{\infty}}\right)\nonumber\\
		=&-q(1+q^2+q^4)(1+q^3+q^6)\left(\frac{(q^{12};q^3)_{\infty}}{(1-q^2)(q^5;q)_{\infty}}-\frac{q^3(q^{12};q^3)_{\infty}}{(1-q^2)(q^4;q)_{\infty}}\right).	
	\end{align}
Combining \eqref{eq-b32n-b31n} and \eqref{eq-frac-q3-q3}, we derive our desired result.
\end{proof}

We proceed to show that the coefficients of $q^n$ {in}
\begin{equation}\label{eq-mainle}
	\frac{(q^{12};q^3)_{\infty}}{(1-q^2)(q^5;q)_{\infty}}- \frac{q^3(q^{12};q^3)_{\infty}}{(1-q^2)(q^4;q)_{\infty}}
\end{equation}
are non-positive for all $n\ge 152$. Thus by Lemma \ref{thm1-lem-1}, we see that $b_{3,2}(n)\ge b_{3,1}(n)$ for $n\ge 163$.

We first   give combinatorial interpretations of
\[\frac{(q^{12};q^3)_{\infty}}{(1-q^2)(q^5;q)_{\infty}}\quad\text{and}
\quad \frac{q^3(q^{12};q^3)_{\infty}}{(1-q^2)(q^4;q)_{\infty}}\]
 as follows.

On the one hand, let $\mathcal{T}(n)$ denote the set of partition triplets $(\alpha,\beta,\gamma)$ where $\alpha$ is a partition with each part congruent to $2$ modular $3$, $\beta$ is a partition with each part congruent to $1$ modular $3$ and the minimum part {in $\beta$} is not less than $7$, $\gamma$ is a partition with  each part equal to either $6$ or $9$. Moreover, $|\alpha|+|\beta|+|\gamma|=n$. Let $t(n)$ denote the cardinality of $\mathcal{T}(n)$. It is clear that
\begin{equation}
	\sum_{n=0}^{\infty}t(n)q^n=\frac{1}{(q^2;q^3)_\infty}\cdot\frac{1}{(q^7;q^3)_\infty}\cdot\frac{1}{(1-q^6)(1-q^9)}=\frac{(q^{12};q^3)_{\infty}}{(1-q^2)(q^5;q)_{\infty}}.
\end{equation}

On the other hand, define $\mathcal{S}(n)$ to be the set of partition triplets $(\pi,\mu,\delta)$ where $\pi$ is a partition with each part congruent to $2$ modular $3$, $\mu$ is a partition with each part congruent to $1$ modular $3$ and the minimum part {in $\mu$} is not less than $4$, $\delta$ is a partition with  each part equal to either $6$ or $9$. Moreover $|\pi|+|\mu|+|\delta|=n$. Let $s(n)$ denote the cardinality of $\mathcal{S}(n)$, then we have
\begin{equation}
\sum_{n=0}^{\infty}s(n)q^n=\frac{1}{(q^2;q^3)_\infty}\cdot\frac{1}{(q^4;q^3)_\infty}\cdot\frac{1}{(1-q^6)(1-q^9)}=\frac{(q^{12};q^3)_{\infty}}{(1-q^2)(q^4;q)_{\infty}}\label{eq-gensn}.
\end{equation}

Thus, to show the coefficients of $q^n$ in \eqref{eq-mainle} is non-positive for $n\ge 152$, it suffices to show that $t(n)\le s(n-3)$ for $n\ge 152$. To this end, we first partition the set $\mathcal{T}(n)$ into seven disjoint subsets $\mathcal{T}_i(n)$, where $1\le i\le 7$. We then  list seven disjoint subsets $\mathcal{S}_i(n-3)$ of the set   $\mathcal{S}(n-3)$, where $1\le i\le 7$. After showing that $\# \mathcal{T}_i(n)\le \#\mathcal{S}_i(n-3)$ for $1\le i\le 7$ and $n\ge 152$, we arrive at $t(n)\le s(n-3)$ for $n\ge 152$.

First, $\mathcal{T}_i(n)$ ($1\le i\le 7$) can be described as follows.
\begin{itemize}
	\item[(1)] $\mathcal{T}_1(n)$ is the set of  $(\alpha, \beta, \gamma)\in\mathcal{T}(n)$ such that  $\beta \neq \emptyset$;
	
	\item[(2)] $\mathcal{T}_2(n)$ is the set of $(\alpha, \beta, \gamma)\in\mathcal{T}(n)$ such that $\beta=\emptyset$, $\alpha_1>\alpha_2$ and $\alpha_1 \ge 5$;
	
	\item[(3)] $\mathcal{T}_3(n)$ is the set of $(\alpha, \beta, \gamma)\in\mathcal{T}(n)$ such that $\beta=\emptyset$ and $\alpha_1=\alpha_2\ge 11$;
	
	\item[(4)] $\mathcal{T}_4(n)$ is the set of $(\alpha, \beta, \gamma)\in\mathcal{T}(n)$ such that $\beta=\emptyset$ and $\alpha_1=\alpha_2= 8$;
	
	\item[(5)] $\mathcal{T}_5(n)$ is the set of $(\alpha, \beta, \gamma)\in\mathcal{T}(n)$ such that $\beta=\emptyset$, $f_\gamma(6)\ge 4$ and either $\alpha_1=\alpha_2\le 5$ or $\alpha=(2)$;
	
	\item[(6)] $\mathcal{T}_6(n)$ is the set of $(\alpha, \beta, \gamma)\in\mathcal{T}(n)$ such that $\beta=\emptyset$, $f_\gamma(6)\le 3$, $f_\gamma(9)\ge 2$ and either $\alpha_1=\alpha_2\le 5$ or $\alpha=(2)$;

	\item[(7)] $\mathcal{T}_7(n)$ is the set of $(\alpha, \beta, \gamma)\in\mathcal{T}(n)$ such that $\beta=\emptyset$,  $f_{\gamma}(6)\le 3$, $f_{\gamma}(9)\le 1$ and either $\alpha_1=\alpha_2\le 5$ or $\alpha=(2)$.

\end{itemize}
 It is easy to check that $\cup_{i=1}^7\mathcal{T}_i(n)=\mathcal{T}(n)$ and $\mathcal{T}_i(n)\cap\mathcal{T}_j(n)=\emptyset$ for all $1\le i< j\le 7$.

The seven disjoint subsets $\mathcal{S}_{i}(n-3)$ ($1\le i\le 7$) {are} defined as follows.
\begin{itemize}
	\item[(1)] $\mathcal{S}_1(n-3)$ is the set of $(\pi,\mu,\delta)\in \mathcal{S}(n-3)$ such that the smallest part of $\mu$ is unique;
	
	\item[(2)] $\mathcal{S}_2(n-3)$ is the set of $(\pi,\mu,\delta)\in\mathcal{S}(n-3)$ such that $\pi_1 \ge 2$ and $\mu = \emptyset$;

	\item[(3)] $\mathcal{S}_3(n-3)$ is the set of $(\pi,\mu,\delta)\in\mathcal{S}(n-3)$ satisfying the following three restrictions:
	\begin{itemize}
		\item[(i)] $\mu=(\mu_1,4,4,4)$;
	\item[(ii)] $\mu_1\ge \max\{2\pi_1-15,7\}$;
	\item[(iii)] $\mu_1$ odd;
	\end{itemize}

	\item[(4)] $\mathcal{S}_4(n-3)$ is the set of $(\pi,\mu,\delta)\in\mathcal{S}(n-3)$ such that $\pi_1\le 8$, $5$ is a part of $\pi$ and $\mu=(4,4)$;
	
	\item[(5)] $\mathcal{S}_5(n-3)$ is the set of $(\pi,\mu,\delta)\in\mathcal{S}(n-3)$ such that $\mu=(7,7,7)$ and either $\pi_1=\pi_2\le 5$ or $\pi=(2)$;

	\item[(6)] $\mathcal{S}_6(n-3)$ is the set of $(\pi,\mu,\delta)\in\mathcal{S}(n-3)$ such that $\mu=(7,4,4)$, $f_\delta(6)\le 3$ and either $\pi_1=\pi_2\le 5$ or $\pi=(2)$;
	\item[(7)] $\mathcal{S}_7(n-3)$ is the set of $(\pi,\mu,\delta)\in\mathcal{S}(n-3)$ such that $\pi=(2^x)$, $\mu=(4^y,7^z)$ and $\delta= \emptyset$, where $x,z\ge 0$ and $y\ge 4$.
\end{itemize}
It is easy to check that $\mathcal{S}_i(n-3)$ and $\mathcal{S}_j(n-3)$ are disjoint for any $1\le i<j\le 7$. We proceed to construct six bijections $\phi_i$ between $\mathcal{T}_i(n)$ and $\mathcal{S}_i(n-3)$ for $1\le i\le 6$. Moreover we will show that $\# \mathcal{T}_7(n)\le \# \mathcal{S}_7(n-3)$ for $n\ge 152$. In this way, we verify that $t(n)\le s(n-3)$ for $n\ge 152$.

We first give the bijection $\phi_1$ as follows.

\begin{lem}\label{thm1-lem-2}
	There is a bijection $\phi_1$ between $\mathcal{T}_1(n)$ and $\mathcal{S}_1(n-3)$.
\end{lem}
\begin{proof}
	Let $(\alpha, \beta, \gamma)$ be a  partition triplet in $\mathcal{T}_1(n)$. By definition, we see that $\beta_{\ell(\beta)}\ge 7$.
	
	Define
\begin{equation*}
	(\pi,\mu,\delta)=\phi_1(\alpha,\beta,\gamma):=(\alpha,(\beta_1,\ldots,\beta_{\ell(\beta)-1},\beta_{\ell(\beta)}-3),\gamma).
\end{equation*}
Clearly
 $\mu_{\ell(\mu)-1}=\beta_{\ell(\beta)-1} > \beta_{\ell(\beta)}-3 =\mu_{\ell(\mu)}$, which implies the smallest part of $\mu$ is unique. Moreover
\begin{equation*}
	|\pi|+|\mu|+|\delta|=|\alpha|+|\beta|-3+|\gamma|=n-3.
\end{equation*}
Thus we get $(\pi,\mu,\delta)\in \mathcal{S}_1(n)$. To show $\phi_1$ is a bijection, we need to build the inverse map $\phi_1^{-1}$ from $\mathcal{S}_1(n-3)$ to $\mathcal{T}_1(n)$. Let $(\pi,\mu,\delta)$ be a partition triplet in $\mathcal{S}_1(n-3)$. Define $\phi_1^{-1}$ as follows.
\begin{equation*}
	(\alpha, \beta, \gamma)=\phi_1^{-1}(\pi,\mu,\delta):=(\pi,(\mu_1, \ldots, \mu_{\ell(\mu)-1},\mu_{\ell(\mu)}+3),\delta).
\end{equation*}
It is not difficult to check that $\phi_1^{-1}(\alpha, \beta, \gamma)\in \mathcal{T}_1(n)$ and $\phi_1^{-1}$ is the inverse map of $\phi_1$. This completes the proof.
\end{proof}
For example, let
\[(\alpha,\beta,\gamma)=(\emptyset,(10,10,7),(9,6))\in \mathcal{T}_1(42).\]
Applying $\phi_1$ on $(\alpha,\beta,\gamma)$, we get
\[
(\pi,\mu,\delta)=(\emptyset, (10,10,4),(9,6))\in \mathcal{S}_1(39).
\]
Applying $\phi_1^{-1}$ on $(\pi,\mu,\delta)$, we recover $(\alpha,\beta,\gamma)$.

Next we describe the bijection $\phi_2$ between $\mathcal{T}_2(n)$ and $\mathcal{S}_2(n-3)$.
\begin{lem}\label{thm1-lem-3}
	There is a bijection $\phi_2$ between $\mathcal{T}_2(n)$ and $\mathcal{S}_2(n-3)$.
\end{lem}
\begin{proof}
Given $(\alpha,\beta,\gamma)\in \mathcal{T}_2(n)$, by definition, $\beta=\emptyset$, $\alpha_1>\alpha_2$ and $\alpha_1\ge 5$. Define
\begin{equation*}
	(\pi,\mu,\delta)=\phi_2(\alpha,\beta,\gamma):=((\alpha_1-3,\alpha_2\ldots,\alpha_{\ell(\alpha)}),\emptyset,\gamma).
\end{equation*}
It is routine to check that $(\pi,\mu,\delta)\in \mathcal{S}_2(n-3)$.

Next we verify that $\phi_2$ is a bijection. For any $(\pi,\mu,\delta)\in \mathcal{S}_2(n-3)$, by definition $\mu=\emptyset.$ Define the inverse map $\phi_2^{-1}$ as follows.
\begin{equation*}
	(\alpha,\beta,\gamma)=\phi_2^{-1}(\pi,\mu,\delta):=((\pi_1+3,\pi_2,\ldots,\pi_{\ell(\pi)}),\emptyset, \delta).
\end{equation*}
It is easy to check that $(\alpha,\beta,\gamma)\in \mathcal{T}_2(n)$, and $\phi_2^{-1}$ is indeed the inverse map of $\phi_2$. Thus we deduce $\phi_2$ is a bijection.
\end{proof}
For instance, let
\[(\alpha,\beta,\gamma)=((5,2,2),\emptyset,(9))\in \mathcal{T}_2(18).\]
Applying $\phi_2$ on $(\alpha,\beta,\gamma)$, we derive
\[
(\pi,\mu,\delta)=((2,2,2), \emptyset,(9))\in \mathcal{S}_2(15).
\]
Applying $\phi_2^{-1}$ on $(\pi,\mu,\delta)$, we recover $(\alpha,\beta,\gamma)$.

We proceed to describe the bijection $\phi_3$.
\begin{lem}\label{thm1-lem-4}
	There is a bijection $\phi_3$ between $\mathcal{T}_3(n)$ and $\mathcal{S}_3(n-3)$.
\end{lem}
\begin{proof}
	Given $(\alpha,\beta,\gamma)\in \mathcal{T}_3(n)$. From the definition of $\mathcal{T}_3(n)$, we find that $\alpha_1=\alpha_2\ge 11$ and $\beta=\emptyset$. {Assuming} that $\alpha_1=3k+2$, from $\alpha_1\ge 11$ we find that $k\ge 3$. The map $\phi_3$ can be defined as follows:
	\begin{equation}
		(\pi,\mu,\delta)=\phi_3(\alpha,\beta,\gamma)
		:=((\alpha_3, \ldots, \alpha_{\ell(\alpha)}), (6k-11, 4, 4, 4), \gamma).
	\end{equation}
	Since $k\ge 3$, we get $\mu_1=6k-11\ge 7$ and $\mu_1$ {is} odd. Moreover, from $\pi_1=\alpha_3\le 3k+2$, we find that $\mu_1=6k-11=2(3k+2)-15\ge 2\pi_1-15$. Furthermore,
	\begin{equation}
|\pi|+|\mu|+|\delta|=|\alpha|-\alpha_1-\alpha_2+6k-11+12+|\gamma|=|\alpha|-3+|\gamma|=n-3.
	\end{equation}
	Thus we deduce that $(\pi,\mu,\delta)\in \mathcal{S}_3(n-3)$. Conversely, given $(\pi,\mu,\delta)\in \mathcal{S}_3(n-3)$ where $\ell(\mu)=4$, $\mu_1\ge 7$ odd,  $\mu_2=\mu_3=\mu_4=4$ and $\mu_1\ge 2\pi_1-15$. We define $\phi_3^{-1}$ as follows.
	\begin{equation*}
		(\alpha, \beta, \gamma)=\phi_3^{-1}(\pi,\mu,\delta):=((\frac{\mu_1+15}{2}, \frac{\mu_1+15}{2}, \pi_1, \ldots, \pi_{\ell(\pi)}),\emptyset,\delta).
	\end{equation*}
Clearly $\alpha_1=\alpha_2=\frac{\mu_1+15}{2}\ge\pi_1 =\alpha_3$ and $\alpha_1\ge 11$. Moreover, we have $|\alpha|+|\beta|+|\gamma|=|\pi|+|\mu|+|\delta|+15-12=n-3+3=n$. Thus $(\alpha,\beta,\gamma)\in \mathcal{T}_3(n)$ and it can be checked that $\phi_3^{-1}$ is the inverse map of $\phi_3$,  which implies $\phi_3$ is a bijection.
\end{proof}
For example, let
\[(\alpha,\beta,\gamma)=((11,11,5),\emptyset,(6))\in \mathcal{T}_3(33).\]
Applying $\phi_3$ on $(\alpha,\beta,\gamma)$, we get
\[
(\pi,\mu,\delta)=((5), (7,4,4,4),(6))\in \mathcal{S}_3(30).
\]
Applying $\phi_3^{-1}$ on $(\pi,\mu,\delta)$, we recover $(\alpha,\beta,\gamma)$.

We next present the bijection $\phi_4$ between $\mathcal{T}_4(n)$ and $\mathcal{S}_4(n-3)$.
\begin{lem}\label{thm1-lem-5}
	There is a bijection $\phi_4$ between $\mathcal{T}_4(n)$ and $\mathcal{S}_4(n-3)$.
\end{lem}
\begin{proof}
	Given $(\alpha, \beta, \gamma) \in \mathcal{T}_4(n)$, by definition we have $\alpha_1=\alpha_2=8$ and $\beta=\emptyset$. Let $k$ be the maximum integer such that $\alpha_k\ge 5$. Clearly   $k\ge 2$. Define $\phi_4$ as follows
	\begin{equation*}
	(\pi,\mu,\delta)=\phi_4(\alpha, \beta, \gamma):=((\alpha_3,\ldots,\alpha_k,5,\alpha_{k+1},\ldots,\alpha_{\ell(\alpha)}), (4,4),\gamma).
	\end{equation*}
	It is easy to check that $|\pi|+|\mu|+|\delta|=|\alpha|-\alpha_1-\alpha_2+ 5+8+|\gamma|=|\alpha|+|\beta|+|\gamma|-3=n-3$. Thus $(\pi,\mu,\delta)\in \mathcal{S}_4(n-3)$. Conversely, given $(\pi,\mu,\delta)\in \mathcal{S}_4(n-3)$, by definition, $\mu=(4,4)$ and there exists $t$ such that $\pi_t=5$.   Define the inverse map $\phi_4^{-1}$ as follows.
	\begin{equation*}
		(\alpha, \beta, \gamma)=\phi_4^{-1}(\pi,\mu,\delta):=((8,8,\pi_1,\ldots,\pi_{t-1},\pi_{t+1},\ldots,\pi_{\ell(\pi)}),\emptyset,\gamma).
	\end{equation*}
	It is easy to check that $|\alpha|+|\beta|+|\gamma|=|\pi|-5+16+|\mu|-8+|\delta|=n-3+3=n$ and $(\alpha, \beta, \gamma)\in \mathcal{T}_4(n)$. Moreover $\phi_4^{-1}$ is the inverse map of $\phi_4$. Thus $\phi_4$ is a bijection.
\end{proof}
	For instance, let
	\[(\alpha,\beta,\gamma)=((8,8,8,2),\emptyset,(6))\in \mathcal{T}_4(32).\]
	Applying $\phi_4$ on $(\alpha,\beta,\gamma)$, we get
	\[
	(\pi,\mu,\delta)=((8,5,2), (4,4),(6))\in \mathcal{S}_4(29).
	\]
	Applying $\phi_4^{-1}$ on $(\pi,\mu,\delta)$, we recover $(\alpha,\beta,\gamma)$.

We next establish the bijection $\phi_5$ between $\mathcal{T}_5(n)$ and $\mathcal{S}_5(n-3)$.
\begin{lem}\label{thm1-lem-6}
There is a bijection $\phi_5$ between $\mathcal{T}_5(n)$ and $\mathcal{S}_5(n-3)$.
\end{lem}
\begin{proof}
	Given $(\alpha,\beta,\gamma)\in\mathcal{T}_5(n)$, by definition we have  $\beta=\emptyset$, $\gamma_{\ell(\gamma)-3}=\gamma_{\ell(\gamma)-2}=\gamma_{\ell(\gamma)-1}=\gamma_{\ell(\gamma)}=6$ and either $\alpha_1=\alpha_2\le 5$ or $\alpha=(2)$. We define $\phi_5$ as follows.
	\begin{equation*}
		(\pi,\mu,\delta)=\phi_5(\alpha,\beta,\gamma):=(\alpha, (7,7,7), (\gamma_1,\ldots,\gamma_{\ell(\gamma)-4})).		
	\end{equation*}
Clearly,  $\mu=(7,7,7)$ and either $\pi_1=\pi_2\le 5$ or $\pi=(2)$. Moreover, $|\pi|+|\mu|+|\delta|=|\alpha|+21+|\gamma|- 24=|\alpha|+|\gamma|-3=n-3$. Thus $(\pi,\mu,\delta)\in \mathcal{S}_5(n-3)$.
	
	Next we verify that $\phi_5$ is a bijection. Given $(\pi,\mu,\delta)\in \mathcal{S}_5(n-3)$, we see that $\mu=(7,7,7)$ and either $\pi_1=\pi_2\le 5$ or $\pi=(2)$. The map $\phi_5^{-1}$ can be defined as follows.
	\begin{equation*}
	(\alpha,\beta,\gamma)=\phi_5^{-1}(\pi,\mu,\delta):=(\pi, \emptyset, (\delta_1,\ldots,\delta_{\ell(\delta)},6,6,6,6)).
	\end{equation*}
Moreover, it is not difficult to check that $|\alpha|+|\beta|+|\gamma|=|\pi|-21+|\delta|+24=|\pi|+|\mu|+|\delta|+3=n-3+3=n$, $(\alpha,\beta,\gamma)\in \mathcal{T}_5(n)$ and $\phi_5^{-1}$ is the inverse map of $\phi_5$. This implies $\phi_5$ is a bijection.
\end{proof}
For example, let
\[(\alpha,\beta,\gamma)=((5,5,2),\emptyset,(9,6,6,6,6))\in \mathcal{T}_5(45).\]
Applying $\phi_5$ on $(\alpha,\beta,\gamma)$, we get
\[
(\pi,\mu,\delta)=((5,5,2), (7,7,7),(9))\in \mathcal{S}_5(42).
\]
Applying $\phi_5^{-1}$ on $(\pi,\mu,\delta)$, we recover $(\alpha,\beta,\gamma)$.

We next describe the bijection $\phi_6$ between $\mathcal{T}_6(n)$ and $\mathcal{S}_6(n-3)$.
\begin{lem}\label{thm1-lem-7}
	There is a bijection $\phi_6$ between $\mathcal{T}_6(n)$ and $\mathcal{S}_6(n-3)$.
\end{lem}
\begin{proof}
Given $(\alpha,\beta,\gamma) \in \mathcal{T}_6(n)$ where $\beta=\emptyset$,  $\gamma_1=\gamma_2=9$ and either $\alpha_1=\alpha_2\le 5$ or $\alpha=(2)$. The map $\phi_6$ can be defined as follows.
\begin{equation*}
(\pi,\mu,\delta)=\phi_6(\alpha,\beta,\gamma):=(\alpha,(7,4,4),(\gamma_3,\ldots,\gamma_{\ell(\gamma)})).
\end{equation*}
It is easy to see that $|\pi|+|\mu|+|\delta|=|\alpha|+|\beta|+|\gamma|-18+15=|\alpha|+|\beta|+|\gamma|-3=n-3$. Thus $(\pi,\mu,\delta)\in \mathcal{S}_6(n-3)$. We next show that $\phi_6$ is a bijection. To this end, we describe the inverse map $\phi_6^{-1}$ as follows. Given $(\pi,\mu,\delta)\in \mathcal{S}_6(n-3)$, by definition, $\mu=(7,4,4)$, $f_\delta(6)\le 3$ and either $\pi_1=\pi_2\le 5$ or $\pi=(2)$. Define the map $\phi_6^{-1}$ as given below.
\begin{equation*}
	(\alpha,\beta,\gamma)=\phi_6^{-1}(\pi,\mu,\delta):=(\pi,\emptyset,(9,9,\delta_1,\ldots,\delta_{\ell(\delta)})).
\end{equation*}
Clearly $\beta=\emptyset$. Moreover, $\gamma_1=\gamma_2=9$ and  $|\alpha|+|\beta|+|\gamma|=|\pi|+|\mu|+|\delta|+18-15=n-3+3=n$. Thus we have $(\alpha,\beta,\gamma)\in \mathcal{T}_6(n)$. It is routine to verify that $\phi_{6}^{-1}$ is the inverse map of $\phi_6$. Thus $\phi_6$ is a bijection.
\end{proof}
For example, let
\[(\alpha,\beta,\gamma)=((2),\emptyset,(9,9,9,6))\in \mathcal{T}_6(35).\]
Applying $\phi_6$ on $(\alpha,\beta,\gamma)$, we derive
\[
(\pi,\mu,\delta)=((2), (7,4,4),(9,6))\in \mathcal{S}_6(32).
\]
Applying $\phi_6^{-1}$ on $(\pi,\mu,\delta)$, we recover $(\alpha,\beta,\gamma)$.

We proceed to prove that $\# \mathcal{T}_7(n)\le \# \mathcal{S}_7(n-3)$ for $n\ge 152$. To this end, we first prove the following two lemmas.

\begin{lem}\label{lem-main1-axby}
	For any positive integer $a,b$ such that $(a,b)=1$, let $t_{a,b}(n)$ denote the number of non-negative integer solutions of $ax+by=n$. Then $t_{a,b}(n)$ satisfies the following inequality:	
	\begin{equation}\label{equ-lem-abn}
		\left\lfloor\frac{\left\lfloor\frac{n}{a}\right\rfloor+1}{b}\right\rfloor\le	t_{a,b}(n)\le \left\lceil\frac{\left\lfloor\frac{n}{a}\right\rfloor+1}{b}\right\rceil.
	\end{equation}
\end{lem}

\begin{proof}
	Clearly, $0\le x\le \left\lfloor\dfrac{n}{a}\right\rfloor$, which means there are $\left\lfloor\dfrac{n}{a}\right\rfloor+1$ choices of $x$. Moreover,  since $(a,b)=1$, among every $b$ consecutive integers there exists exactly one integer $x$ such that $y=\dfrac{n-ax}{b}$ is an integer. Hence every $b$ consecutive integers $x$ among $0\le x\le \left\lfloor\dfrac{n}{a}\right\rfloor$ contributes exactly one to $t_{a,b}(n)$. This yields \eqref{equ-lem-abn}.
\end{proof}

\begin{lem}\label{lem-main1-tabc}
	Given positive integer $a,b,c$ such that $(a,b)=1$. Let $t_{a,b,c}(n)$ denote the number of non-negative integer solutions of the equation $ax+by+cz=n$. Then we have
	\begin{equation}
		t_{a,b,c}(n)\ge \frac{n^2}{2abc}-(\frac{n}{c}+1).
	\end{equation}
\end{lem}
\begin{proof}
	From Lemma \ref{lem-main1-axby}, we have \begin{equation}\label{equ-tab-lower}
		t_{a,b}(n)\ge \left\lfloor\frac{\left\lfloor\frac{n}{a}+1\right\rfloor}{b}\right\rfloor\ge \frac{n}{ab}-1.
	\end{equation}
	Clearly $0\le z\le \left\lfloor\dfrac{n}{c}\right\rfloor$. Moreover, for fixed $0\le z\le \left\lfloor\dfrac{n}{c}\right\rfloor$, the number of non-negative integer solutions of $ax+by+cz=n$ equals $t_{a,b}(n-cz)$. Using \eqref{equ-tab-lower} we have	
	\begin{align}\label{equ-tabc-n} t_{a,b,c}(n)=&\sum_{z=0}^{\left\lfloor\frac{n}{c}\right\rfloor} t_{a,b}(n-cz)
		\ge  \sum_{z=0}^{\left\lfloor\frac{n}{c}\right\rfloor}
		\left(\frac{n-cz}{ab}-1\right).
	\end{align}
	We next calculate $\sum_{z=0}^{\left\lfloor\frac{n}{c}\right\rfloor}
	\frac{n-cz}{ab}
	$. Assume $n=ck+i$, where $0\le i\le c-1$. Then we have
	\begin{align}\label{equ-n-ck+i}
		\sum_{z=0}^{\left\lfloor\frac{n}{c}\right\rfloor}
		\frac{n-cz}{ab}
		&=\sum_{z=0}^{k}
		\frac{n-cz}{ab}
		\nonumber\\
		&=\frac{\sum_{z=0}^{k}(i+cz)}{ab} \nonumber\\
		&=\frac{2i(k+1)+ck(k+1)}{2ab}\nonumber\\
		&=\frac{(n+i)(k+1)}{2ab}.
	\end{align}
	Note that $k\ge \frac{n}{c}-1$, we have
	\begin{equation}\label{eq-est-1}
		\sum_{z=0}^{\left\lfloor\frac{n}{c}\right\rfloor}
		\frac{n-cz}{ab}\ge \frac{(n+i)n}{2abc}\ge \frac{n^2}{2abc}.
	\end{equation}
	Combining \eqref{equ-tabc-n} and \eqref{eq-est-1}, we derive our desired result.
\end{proof}
  Now we prove $\# \mathcal{T}_7(n)\le \# \mathcal{S}_7(n-3)$ when $n\ge 152$ as stated in the following lemma.
\begin{lem}\label{thm1-lem-8}
	For $n\ge 152$,
	\begin{equation}\label{eq-T7S7}
		\# \mathcal{T}_7(n)\le \# \mathcal{S}_7(n-3).
	\end{equation}
\end{lem}
\begin{proof}
	Given $(\alpha,\beta,\gamma)\in \mathcal{T}_7(n)$, by definition, $\beta=\emptyset$, the part of $\alpha$ can only be either $5$ or $2$ and the part of $\gamma$  can only be either $9$ or $6$. Moreover, $f_\gamma(9)\le 1$, $f_{\gamma}(6)\le 3$ and $f_\alpha(5)\ne 1$. So $\gamma$ is one of the following eight partitions.
	\begin{equation*}
\emptyset,(9),(6),(9,6),(6,6),(9,6,6),(6,6,6),(9,6,6,6).
	\end{equation*}
For each $\gamma$, the number of $\alpha$ satisfying $|\alpha|=n-|\gamma|$ is equal to the number of non-negative integer solutions of $5a+2b=n-|\gamma|$, where $a\ne 1$. Hence $\#\mathcal{T}_7(n)$ equals the sum of the numbers of non-negative integer solutions of the following eight equations:
	\begin{align*}
		&5a+2b=n;\\
		&5a+2b=n-6;\\
		&5a+2b=n-9;\\
		&5a+2b=n-12;\\
		&5a+2b=n-15;\\
		&5a+2b=n-18;\\
		&5a+2b=n-21;\\
		&5a+2b=n-27,
	\end{align*}
	where $ a\ne 1$.
	
	Let $t(n)$ denote the number of non-negative integer solutions of $5a+2b=n$, where $a\ne 1$. From the above analysis, we find that
	\begin{equation}\label{equ-T7n} \#\mathcal{T}_7(n)=\sum_{i=0}^{1}\sum_{j=0}^{3}t(n-9i-6j).
	\end{equation}
	We next find an upper bound for $t(n)$. When $a=0$, clearly there are at most one integer solution of $2b=n$ depending on the parity of $n$. When $a\ge 2$, set $c=a-2$ and $d=b$, we see that $5a+2b=n$ is equivalent to $5c+2d=n-10$, where $c,d$ are non-negative integers.  By Lemma \ref{lem-main1-axby}, we have
\begin{equation}\label{eq-set-tn}
  t(n)\le \frac{\frac{n-10}{2}+1}{5}+1=\frac{n+2}{10}.
\end{equation}
Clearly $\frac{n+2}{10}$  increases with  $n$. Hence by \eqref{equ-T7n} and \eqref{eq-set-tn},
\begin{equation}\label{eq-est-T7n}
  \#\mathcal{T}_7(n)\le \sum_{i=0}^{1}\sum_{j=0}^{3}\frac{n+2}{10}\le
\frac{8(n+2)}{10}.
\end{equation}

	On the other hand, we proceed to find a lower bound for $\#\mathcal{S}_7(n-3)$. Given $(\pi,\mu,\delta)\in \mathcal{S}_7(n-3)$, by definition, $\pi=(2^x)$, $\mu=(4^y,7^z)$ and $\delta= \emptyset$, where $y\ge 4$. Setting $a=x,$ $b=y-4$ and $c=z$, it is easy to check that $\#\mathcal{S}_7(n-3)$ equals the number of the non-negative integer solutions of the following equation.
	\begin{equation}\label{eq-2x+7y+4z}
		2a+4b+7c=n-19,
	\end{equation}
	where $a,b,c \in \mathbb{N}$. Then using Lemma \ref{lem-main1-tabc}, we arrive at
\begin{equation}
	\# \mathcal{S}_7(n-3)\ge \frac{(n-19)^2}{2\cdot2\cdot7\cdot4}-\left(\frac{n-19}{4}+1\right)=\frac{(n-19)^2}{112}-\frac{n-15}{4}.
\end{equation}
It can be checked that when $n\ge 152$, the following inequality holds.
	\begin{equation}
		8\left(\frac{n+2}{10}\right)\le \frac{(n-19)^2}{112}-\frac{n-15}{4}.
	\end{equation}
	Thus we deduce \eqref{eq-T7S7} holds for $n\ge 152$.
\end{proof}

\noindent{\it{Proof of Theorem \ref{thm-main1}}.}
 From Lemma \ref{thm1-lem-2}-Lemma \ref{thm1-lem-7} and Lemma \ref{thm1-lem-8}, we deduce that the coefficients of $q^n$ in \eqref{eq-mainle} are non-positive  for $n\ge 152$. Thus by \eqref{eq-frac-q3-q3}, we deduce that $b_{3,2}(n)\ge b_{3,1}(n)$ when $n\ge 163$. It is trivial to check that the coefficients of $q^n$ in \eqref{eq-b32n-b31n} are non-negative when $28\le n\le 162$. This completes the proof.
\qed

\section{Proof of Theorem \ref{thm-mainodd}}\label{section-proofmainodd}
In this section, we will give a proof of Theorem \ref{thm-mainodd}. To this end, we first recall the generating function of $b_{2,i}(n)$ as stated in Theorem \ref{thm-ahq} due to Craig, Dawsey and Han \cite{craig-dawsey-han-inequalities}. In Lemma \ref{lem-main3-1}, we will show that
$\dfrac{1-q}{(-q;q)_{\infty}}\left(\sum_{n=0}^{\infty}(b_{2,2t+1}(n)-b_{2,2t+2}(n))q^n\right)$ can be written as $\dfrac{A_t(q)}{B_t(q)}$, where $A_t(q)$ is a polynomial of $q$, and 	$B_t(q)=\prod_{i=0}^{2t+1}\left(1+q+\cdots+q^{2i+1}\right)$  has positive coefficients. Thus  it suffices to prove  $B_t(q)\left(\sum_{n=0}^{\infty}(b_{2,2t+1}(n)-b_{2,2t+2}(n))q^n\right)=\dfrac{A_t(q)(-q;q)_\infty}{(1-q)}$ has infinite negative terms. Corollary \ref{cor-main3-1} shows that $A_t(1)<0$. Moreover, let $S(n)$ denote the coefficient of $q^n$ in $\dfrac{(-q;q)_\infty}{1-q}$, we will show that $S(n-k)\sim S(n)$ for any fixed $k$ as $n\to \infty$ in Lemma \ref{lem-main3-4}. Combining Corollary \ref{cor-main3-1} and Lemma \ref{lem-main3-4}, we are able to prove that the coefficients of $q^n$ in $\dfrac{A_t(q)(-q;q)_\infty}{(1-q)}$ are negative as $n\rightarrow \infty.$ This gives the proof of Theorem \ref{thm-mainodd}.

 Recall  that Craig, Dawsey and Han \cite{craig-dawsey-han-inequalities} found the following expression of the generating function on $b_{2,i}(n)$.

\begin{thm}[\cite{craig-dawsey-han-inequalities}]\label{thm-ahq}
	For any $i\ge 1$,
	\[\frac{\sum_{n=0}^{\infty}b_{2,i}(n)q^n}{(-q;q)_\infty}\]
	is a rational function of $q$. To be specific,
	\begin{align}\label{eq-ahq}
		\sum_{n=0}^{\infty}b_{2,i}(n)q^n
		=&(-q;q)_{\infty}\left(\sum_{j=0}^{\lceil i/2\rceil-1}\binom{i-j-1}{j}_{q^2}B_1(i,j;q)+\sum_{j=0}^{\lfloor i/2\rfloor-1}\binom{i-j-2}{j}_{q^2}B_2(i,j;q)\right),
	\end{align}
	where
	\begin{align*}
		B_1(i,j;q)=&q^i\sum_{k=0}^{j}\binom{j}{k}_{q^2}\frac{(-1)^kq^{k^2}}{1-q^{2i+2k-4j}},\\
		B_2(i,j;q)=&q^{3i-4j-3}\sum_{k=0}^{j}\binom{j}{k}_{q^2}\frac{(-1)^kq^{k^2+2k}}{1-q^{2i+2k-4j-2}}.
	\end{align*}
\end{thm}

We first transform
\[\frac{1-q}{(-q;q)_{\infty}}\left(\sum_{n=0}^{\infty}(b_{2,2t+1}(n)-b_{2,2t+2}(n))q^n\right)\]
 as in the following lemma.

\begin{lem}\label{lem-main3-1}
	\begin{equation}\label{equ-frac-1-q}
		\frac{1-q}{(-q;q)_{\infty}}\left(\sum_{n=0}^{\infty}(b_{2,2t+1}(n)-b_{2,2t+2}(n))q^n\right)=\frac{A_{t}(q)}{B_{t}(q)},
	\end{equation}
	where $A_t(q)$ is a polynomial of $q$ and 	\[B_t(q)=\prod_{i=0}^{2t+1}\left(1+q+\cdots+q^{2i+1}\right)\]  has positive coefficients.
\end{lem}
\begin{proof}
	From Theorem \ref{thm-ahq}, we see that
	\begin{align}
		&\frac{1-q}{(-q;q)_{\infty}}\left(\sum_{n=0}^{\infty}(b_{2,2t+1}(n)-b_{2,2t+2}(n))q^n\right)\nonumber\\
		=&\sum_{j=0}^{t}\binom{2t-j}{j}_{q^2}q^{2t+1}\sum_{k=0}^{j}\binom{j}{k}_{q^2}\frac{(-1)^kq^{k^2}}{1+q+\cdots+q^{4t+2k-4j+1}}\nonumber\\
		&\left.+\sum_{j=0}^{t-1}\binom{2t-1-j}{j}_{q^2}q^{6t-4j}\sum_{k=0}^{j}\binom{j}{k}_{q^2}\frac{(-1)^kq^{k^2+2k}}{1+q+\cdots+q^{4t+2k-4j-1}}\right.\nonumber\\
		&\left.-\sum_{j=0}^{t}\binom{2t-j+1}{j}_{q^2}q^{2t+2}\sum_{k=0}^{j}\binom{j}{k}_{q^2}\frac{(-1)^kq^{k^2}}{1+q+\cdots+q^{4t+2k-4j+3}}\right.\nonumber\\
		&-\sum_{j=0}^{t}\binom{2t-j}{j}_{q^2}q^{6t-4j+3}\sum_{k=0}^{j}\binom{j}{k}_{q^2}\frac{(-1)^kq^{k^2+2k}}{1+q+\cdots+q^{4t+2k-4j+1}}.\label{equ-43}
	\end{align}
	Clearly
	\[B_t(q)=\prod_{i=0}^{2t+1}\left(1+q+\cdots+q^{2i+1}\right)\]
	is a common multiple of each denominator in \eqref{equ-43}.	This yields our desired result.
\end{proof}

The following lemma gives the value of $\dfrac{A_t(1)}{B_t(1)}$.

\begin{lem}\label{lem-main3-2}
	We have
	\begin{equation}
	\frac{A_t(1)}{B_t(1)}=-\frac{1}{2(2t+1)(2t+2)}.
	\end{equation}
\end{lem}
\begin{proof}
	From \eqref{equ-43}, we have
	\begin{align}\label{equ-45}
		&\frac{A_t(1)}{B_t(1)}\nonumber\\
		=&\lim_{q\to 1}\frac{1-q}{(-q;q)_{\infty}}\left(\sum_{n=0}^{\infty}(b_{2,2t+1}(n)-b_{2,2t+2}(n))q^n\right)\nonumber\\
		=&\sum_{j=0}^{t-1}\binom{2t-1-j}{j}\sum_{k=0}^{j}\binom{j}{k}\frac{(-1)^k}{4t+2k-4j}-\sum_{j=0}^{t}\binom{2t+1-j}{j}\sum_{k=0}^{j}\binom{j}{k}\frac{(-1)^k}{4t+4+2k-4j}\nonumber\\
		=&\frac{1}{2}\left(\sum_{j=0}^{t-1}\binom{2t-1-j}{j}\sum_{k=0}^{j}\binom{j}{k}\frac{(-1)^k}{2t+k-2j}-\sum_{j=0}^{t}\binom{2t+1-j}{j}\sum_{k=0}^{j}\binom{j}{k}\frac{(-1)^k}{2t+2+k-2j}\right).
			\end{align}
			Notice that
			\[\frac{1}{i+1}=\int_{0}^{1}x^i{\rm d}x.\]
			Hence \eqref{equ-45} can be further transformed as follows.
			\begin{align}
	&\frac{A_t(1)}{B_t(1)}\nonumber\\	=&\frac{1}{2}\left(\sum_{j=0}^{t-1}\binom{2t-1-j}{j}\sum_{k=0}^{j}(-1)^k\binom{j}{k}\int_{0}^{1}x^{2t+k-2j-1}\,{\rm d}x\right.\nonumber\\
		&\left.-\sum_{j=0}^{t}\binom{2t+1-j}{j} \sum_{k=0}^{j}(-1)^k\binom{j}{k}\int_{0}^{1}x^{2t+1+k-2j}\,{\rm d}x\right)\nonumber\\
		=&\frac{1}{2}\left(\sum_{j=0}^{t-1}\binom{2t-1-j}{j}\int_{0}^{1}x^{2t-2j-1}\sum_{k=0}^{j}(-1)^k\binom{j}{k}x^k\,{\rm d}x\right.\nonumber\\
		&\left.-\sum_{j=0}^{t}\binom{2t+1-j}{j}\int_{0}^{1}x^{2t+1-2j}\sum_{k=0}^{j}(-1)^k\binom{j}{k}x^k\,{\rm d}x\right)\nonumber\\
		=&\frac{1}{2}\left(\sum_{j=0}^{t-1}\binom{2t-1-j}{j}\int_{0}^{1}x^{2t-2j-1}(1-x)^j\,{\rm d}x-\sum_{j=0}^{t}\binom{2t+1-j}{j}\int_{0}^{1}x^{2t-2j+1}(1-x)^j\,{\rm d}x\right)\label{eq-tointx1-x}.
	\end{align}
	Recall that the well known Eulerian integral of the first kind implies that
	for non-negative integer $a,b$,
	\begin{equation}\label{eq-Euler-Beta}
		\int_{0}^{1}x^a(1-x)^b\,{\rm d}x=\frac{1}{a+b+1}\frac{1}{\binom{a+b}{b}}.
	\end{equation}
	Substituting \eqref{eq-Euler-Beta} into \eqref{eq-tointx1-x}, we have
	\begin{align}
		\frac{A_t(1)}{B_t(1)}
		=&\frac{1}{2}\left(\sum_{j=0}^{t-1}\binom{2t-1-j}{j}\frac{1}{\binom{2t-1-j}{j}(2t-j)}-\sum_{j=0}^{t}\binom{2t+1-j}{j}\frac{1}{\binom{2t+1-j}{j}(2t-j+2)}\right)\nonumber\\
		=&\frac{1}{2}\left(\sum_{j=0}^{t-1}\frac{1}{2t-j}-\sum_{j=0}^{t}\frac{1}{2t-j+2}\right)\nonumber\\
		=&-\frac{1}{2(2t+1)(2t+2)}.
	\end{align}
\end{proof}

From Lemma \ref{lem-main3-2}, the following corollary is clear.

\begin{core}\label{cor-main3-1}
	We have $A_t(1)<0$.
\end{core}
\begin{proof}
	From Lemma \ref{lem-main3-2}, we know that
	\[	\frac{A_t(1)}{B_t(1)}<0.\] Moreover, the coefficients of $B_{t}(q)$ are all non-negative, which implies $ B_t(1)>0$. Thus we have $A_t(1)<0$.
\end{proof}

 Let $S(n)$ denote the coefficient of $q^n$ in $\dfrac{(-q;q)_\infty}{1-q}$ as stated in the beginning of this section, we have the following result on $S(n)$.
\begin{lem}\label{lem-main3-4}
	For fixed $k$, we have $S(n+k)\sim S(n)$ when $n\to \infty$.
\end{lem}
\begin{proof}
	Let $Q(n)$ denote the number of distinct partitions of $n$. Clearly (see \cite[pp. 5]{andrews1998theory})
	\[\sum_{n=0}^{\infty}Q(n)q^n=(-q;q)_\infty.\]
	The asymptotic function of $Q(n)$ is given below (see  \cite{Abramowitz1972}),
\begin{equation}\label{eq-preQ}
	Q(n)\sim \frac{e^{\pi \sqrt{n/3}}}{4\cdot 3^{\frac{1}{4}}\cdot n^{\frac{3}{4}}}.
\end{equation}
From \eqref{eq-preQ}, it is easy to see that $Q(n+k)\sim Q(n)$ for any fixed $k$ as $n\to\infty.$ By the definition of $S(n)$, we have
	\begin{align}
		\sum_{n=0}^{\infty}S(n)q^n&=\frac{(-q;q)_\infty}{1-q}\nonumber\\
		&=\sum_{n=0}^{\infty}Q(n)q^n\sum_{n=0}^{\infty}q^n\nonumber\\
		&=\sum_{n=0}^{\infty}\sum_{i=0}^{n}Q(i)q^n.
	\end{align}
Thus
	\begin{equation}
		S(n)=\sum_{i=0}^{n}Q(i).
	\end{equation}
	Clearly $S(n)\to \infty$ as $n\to \infty$. Using the well known O'Stolz theorem, we have
	\begin{align}\label{eq-QQQQQ}
		\lim_{n\to \infty}\frac{S(n+k)}{S(n)}=&\lim_{n\to \infty}\frac{S(n+k+1)-S(n+k)}{S(n+1)-S(n)}\nonumber\\
		=&\lim_{n\to \infty}\frac{Q(n+k+1)}{Q(n+1)}\nonumber\\
		=&1.
	\end{align}
\end{proof}

We are now in a position to prove Theorem \ref{thm-mainodd}.

\noindent{\it{Proof of Theorem \ref{thm-mainodd}}.}  Since $A_t(q)$ is a polynomial, we may assume that the degree of $A_t(q)$ equals $k$. Moreover, we can write $A_t(q)$ as follows:
\begin{equation}
	A_t(q)=\sum_{i=0}^{k} a_iq^i.
\end{equation}
By Corollary \ref{cor-main3-1}, we see that $\sum_{i=0}^{k} a_i<0$.

From the definition of ${A_t(q)}$ and ${B_t(q)}$, we find that
\begin{align}
	B_{t}(q)\left(\sum_{n=0}^{\infty}(b_{2,2t+1}(n)-b_{2,2t+2}(n))q^n\right)
	=&A_{t}(q)\sum_{n=0}^{\infty}S(n)q^n\nonumber\\
	=&\left(\sum_{n=0}^{k}a_nq^n\right)\cdot\sum_{n=0}^{\infty}S(n)q^n\nonumber\\
	=&\sum_{n=0}^{\infty}\left(\sum_{i=0}^{k}a_i S(n-i)\right)q^n.
\end{align}
Using Corollary \ref{cor-main3-1} and Lemma \ref{lem-main3-4}, when $n\to\infty$,
\[
\sum_{i=0}^{k}a_i S(n-i)\sim S(n) \left(\sum_{i=0}^{k}a_i\right)<0.
\]
Thus
\[
	B_{t}(q)\left(\sum_{n=0}^{\infty}(b_{2,2t+1}(n)-b_{2,2t+2}(n))q^n\right)
\]
has negative coefficients of $q^n$ for sufficiently large $n$. Moreover, since the coefficients  in $B_{t}(q)$ are all non-negative. We derive that $\sum_{n=0}^{\infty}(b_{2,2t+1}(n)-b_{2,2t+2}(n))q^n$ has infinitely many negative terms.

Hence  there are infinitely many $n$ such that $b_{2,k}(n)<b_{2,k+1}(n)$ for $k\ge 3$ and $k$ is odd.\qed

\section{Proof of Theorem \ref{thm-main21} and Theorem \ref{thm-main22}}\label{section-proofmain2}

This section is devoted to proving Theorem \ref{thm-main21} and Theorem \ref{thm-main22}. For Theorem \ref{thm-main21}, we first transform $\sum_{n=0}^{\infty}(b_{2,4}(n)-b_{2,5}(n))q^n$ as the summations of two summands  in Lemma \ref{thm-main21-lem1}. We then show that each summand has non-negative terms for $n\ge 25$. Thus after checking $b_{2,4}(n)\ge b_{2,5}(n)$ for $0\le n\le 24$ and $n\ne 5$, we arrive at Theorem \ref{thm-main21}. The proof of Theorem \ref{thm-main22} is similar. We first rewrite $\sum_{n=0}^{\infty}(b_{2,6}(n)-b_{2,7}(n))q^n$ as the summation of three summands. Next we prove that each summand has non-negative coefficients of $q^n$ for $n\ge 55$. Theorem \ref{thm-main22} follows after we check that $b_{2,6}(n)\ge b_{2,7}(n)$ for $0\le n\le 54$ and $n\ne 7$.

We first transform $\sum_{n=0}^{\infty}(b_{2,4}(n)-b_{2,5}(n))q^n$ as in the following lemma.

\begin{lem}\label{thm-main21-lem1}
	We have
	\begin{equation}\label{eq-b24n-b25nhuajianhou}
		\sum_{n=0}^{\infty}(b_{2,4}(n)-b_{2,5}(n))q^n=\frac{q^4(A(q)(1-q)^2+B(q)(1-q))}{(1-q^8)(1-q^{10})(q;q^2)_{\infty}},
	\end{equation}
	where
	\begin{align}
		A(q)=&2+q+3q^2+3q^3+4q^4+4q^5+6q^6+7q^7+6q^8+7q^9\nonumber\\
		&+5q^{10}+6q^{11}+4q^{12}+4q^{13}+3q^{14}+3q^{15}+q^{16},\label{eq-def-A}\\
		B(q)=&q^{18}+q^{20}.\label{eq-def-B}
	\end{align}
\end{lem}
\begin{proof}
Using Theorem \ref{thm-ahq}, we have
\begin{align}
	&\sum_{n=0}^{\infty}(b_{2,4}(n)-b_{2,5}(n))q^n\nonumber\\
	=&(-q;q)_{\infty}\left(\sum_{j=0}^{1}\binom{4-j-1}{j}_{q^2}B_1(4,j;q)+\sum_{j=0}^{1}\binom{4-j-2}{j}_{q^2}B_2(4,j;q)\right.\nonumber\\
	&\left.-\sum_{j=0}^{2}\binom{5-j-1}{j}_{q^2}B_1(5,j;q)-\sum_{j=0}^{1}\binom{5-j-2}{j}_{q^2}B_2(5,j;q)\right)\nonumber\\
	=&\frac{q^4}{(1-q^8)(1-q^{10})(q;q^2)_{\infty}}\left(2-3q+3q^2-2q^3+q^4-q^5+2q^6-q^7-2q^8+2q^9-3q^{10}\right.\nonumber\\
	&\left.+3q^{11}-3q^{12}+2q^{13}-q^{14}+q^{15}-2q^{16}+q^{17}+2q^{18}-q^{19}+q^{20}-q^{21}\right).
\end{align}
It is routine to check that
\begin{align}\label{eq-AB}
&2-3q+3q^2-2q^3+q^4-q^5+2q^6-q^7-2q^8+2q^9-3q^{10}+3q^{11}-3q^{12}+2q^{13}\nonumber\\	&-q^{14}+q^{15}-2q^{16}+q^{17}+2q^{18}-q^{19}+q^{20}-q^{21}\nonumber\\
=&A(q)(1-q)^2+B(q)(1-q),
\end{align}
where $A(q)$ and $B(q)$ are defined in \eqref{eq-def-A} and \eqref{eq-def-B}, respectively.
Thus we deduce \eqref{eq-b24n-b25nhuajianhou}.	
\end{proof}

We proceed to analyze the positivity of
\begin{equation}\label{pos-frac-1-q^2}
\frac{(1-q)^2}{(1-q^8)(1-q^{10})(q;q^2)_\infty}.\end{equation}
To this end, we need to recall the following result.
\begin{thm}[\cite{WolframQ}]\label{lem-pos-mathworld}
	For $n\ge 5$, the coefficients of $q^n$ in
	\[\frac{(1-q)^2}{(q;q^2)_\infty}\]
	are non-negative.
\end{thm}

We now  state the positivity of \eqref{pos-frac-1-q^2} in the following lemma.

\begin{lem}\label{lem-8-10-q3q2}
  For $n\ge 5$, the coefficients of $q^n$ in
  \begin{equation}\label{eq-main2-1-q1-q81-q10}
  	\frac{(1-q)^2}{(1-q^8)(1-q^{10})(q;q^2)_\infty}
  	\end{equation}
are non-negative.
\end{lem}
\begin{proof}
	 From Theorem \ref{lem-pos-mathworld} and by direct calculation, we find that
\begin{equation*}
	\frac{(1-q)^2}{(q;q^2)_{\infty}}=-q+q^3-q^4+q^9+q^{12}+q^{14}+q^{21}+M(q),
\end{equation*}
where $M(q)$ is a $q$-series with non-negative coefficients.
	Thus \eqref{eq-main2-1-q1-q81-q10} can be expanded as follows.
	\begin{align}\label{eq-main21-1421}
		&\frac{(1-q)^2}{(1-q^8)(1-q^{10})(q;q^2)_\infty}\nonumber\\
		=&\frac{1}{(1-q^8)(1-q^{10})}\left(-q+q^3-q^4+q^9+q^{12}+q^{14}+q^{21}+M(q)\right)\nonumber\\
		=&\frac{1}{(1-q^8)(1-q^{10})}\left((-q+q^3-q^4)(1-q^8)(1-q^{10})+q^{13}+q^{19}+q^{22}+M(q)\right)\nonumber\\
		=&-q+q^3-q^4+\frac{q^{13}+q^{19}+q^{22}+M(q)}{(1-q^8)(1-q^{10})}.
	\end{align}
Hence \eqref{eq-main21-1421} has non-negative coefficients of $q^n$ for $n\ge 5$. This completes the proof.
\end{proof}

We are now in a position to prove Theorem \ref{thm-main21}.

\noindent{\it Proof of Theorem \ref{thm-main21}}.
First,  it is easy to see that
\begin{equation}
	\frac{q^4(1-q)B(q)}{(1-q^8)(1-q^{10})(q;q^2)_{\infty}}=\frac{q^4B(q)}{(1-q^8)(1-q^{10})(q^3;q^2)_{\infty}},
\end{equation}
which obviously has non-negative coefficients.

Moreover, by Lemma \ref{lem-8-10-q3q2} we see that for $n\ge 25$, the coefficients of $q^n$ in
\[\frac{q^4A(q)(1-q)^2}{(1-q^8)(1-q^{10})(q;q^2)_\infty}\]
are non-negative.

Thus from Lemma \ref{thm-main21-lem1}, we find that $b_{2,4}(n)\ge b_{2,5}(n)$ for $n\ge 25$. In addition, it is easy to check that the coefficients of $q^n$ in \eqref{eq-b24n-b25nhuajianhou}  {are} non-negative for $0\le n\le 24$ and $n\ne 5$. Thus \eqref{eq-mainb24ngeb25n} holds for $n\ge 0$ and $n\ne 5$.
\qed

We proceed to prove Theorem \ref{thm-main22}. To this end, we first transform $\sum_{n=0}^{\infty}(b_{2,6}(n)-b_{2,7}(n))q^n$ as in the following lemma.
\begin{lem}\label{thm-main22-lem1}
	\begin{align}\label{eq-b26-b27q6}
		&\sum_{n=0}^{\infty}(b_{2,6}(n)-b_{2,7}(n))q^n\nonumber\\
		=&\frac{q^6C(q)(1-q)^2}{(1-q^{12})(1-q^{14})(q;q^2)_{\infty}}+\frac{q^{16}(1+q)(1-q)^2}{(1-q^{12})(q;q^2)_{\infty}}+\frac{q^{33}(1+q^3)}{(1-q^{12})(1-q^{14})(q^3;q^2)_{\infty}},
	\end{align}
	where
	\begin{align}\label{eq-main22cq}
		C(q)=&3+q+4q^2+2q^3+6q^4+3q^5+7q^6+5q^7+7q^8+6q^9+9q^{10}+9q^{11}+9q^{12}+9q^{13}\nonumber\\&+6q^{14}+9q^{15}+4q^{16}+7q^{17}+2q^{18}+6q^{19}+2q^{20}+4q^{21}+2q^{22}+3q^{23}+q^{28}+q^{30}+q^{32}.
	\end{align}
\end{lem}
\begin{proof}
	Using Theorem \ref{thm-ahq}, we have
\begin{align}
	&\sum_{n=0}^{\infty}(b_{2,6}(n)-b_{2,7}(n))q^n\nonumber\\
	=&(-q;q)_{\infty}\left(\sum_{j=0}^{2}\binom{6-j-1}{j}_{q^2}B_1(6,j;q)+\sum_{j=0}^{2}\binom{6-j-2}{j}_{q^2}B_2(6,j;q)\right.\nonumber\\
	&\left.-\sum_{j=0}^{3}\binom{7-j-1}{j}_{q^2}B_1(7,j;q)-\sum_{j=0}^{2}\binom{7-j-2}{j}_{q^2}B_2(7,j;q)\right)\nonumber\\
	=&\frac{q^6}{(1-q^{12})(1-q^{14})(q^3;q^2)_{\infty}}\left(3-2q+3q^2-2q^3+4q^4-3q^5+4q^6-2q^7\right.\nonumber\\
	&\left.+2q^8-q^9+4q^{10}-q^{12}-3q^{14}+3q^{15}-5q^{16}+3q^{17}-5q^{18}+4q^{19}-4q^{20}\right.\nonumber\\
	&\left.+2q^{21}-2q^{22}+q^{23}-4q^{24}+q^{26}+q^{27}+q^{28}-q^{29}+2q^{30}-q^{31}+q^{32}-q^{33}\right)\label{eq-b26n-b27nhuajianhou}.
\end{align}

One can easily verify that
	\begin{small}
		\begin{align}\label{eq-main22-3-2q}	&3-2q+3q^2-2q^3+4q^4-3q^5+4q^6-2q^7+2q^8-q^9+4q^{10}-q^{12}-3q^{14}+3q^{15}-5q^{16}+3q^{17}\nonumber\\
		&-5q^{18}+4q^{19}-4q^{20}+2q^{21}-2q^{22}+q^{23}-4q^{24}+q^{26}+q^{27}+q^{28}-q^{29}+2q^{30}-q^{31}+q^{32}-q^{33}\nonumber\\
		=&C(q)(1-q)+q^{10}(1-q^{14})(1-q^2)+q^{27}+q^{30},
		\end{align}
	\end{small}
	where $C(q)$ is defined as \eqref{eq-main22cq}. Plugging \eqref{eq-main22-3-2q} into \eqref{eq-b26n-b27nhuajianhou}, we deduce \eqref{eq-b26-b27q6}.
	\end{proof}
Similar to the proof of Theorem \ref{thm-main21}, we next prove that
\[\frac{(1-q)^2}{(1-q^{12})(1-q^{14})(q;q^2)_{\infty}}\] has non-negative coefficients of $q^n$ when $n\ge 17$ in the following lemma.

\begin{lem}\label{lem-main22-1}
	For $n\ge 17$, the coefficients of  $q^n$ in
	\begin{equation}\label{eq-main22-1-q1214}
		\frac{(1-q)^2}{(1-q^{12})(1-q^{14})(q;q^2)_{\infty}}\quad \text{and}\quad \frac{(1-q)^2}{(1-q^{12})(q;q^2)_{\infty}}
	\end{equation}
	are both non-negative.
\end{lem}
\begin{proof}
	Using Theorem \ref{lem-pos-mathworld} and direct calculation, we find that
	\begin{align}\label{eq-main21111}
		\frac{(1-q)^2}{(q;q^2)_{\infty}}
		=-q-q^4+q^{15}+q^{18}+q^{25}+q^{27}+q^{28}+q^{30}+N(q),
		\end{align}
		where $N(q)$ is a $q$-series with non-negative coefficients.
		
	On the one hand, it is easy to check that
		\begin{equation}\label{equ-temp-45}
			-q-q^4+q^{15}+q^{18}+q^{25}+q^{27}+q^{28}+q^{30}=(-q-q^4-q^{13}-q^{16})(1-q^{12})(1-q^{14})+q^{27}+q^{30}+q^{39}+q^{42}.
		\end{equation}
		From \eqref{eq-main21111} and \eqref{equ-temp-45}, we find that
		\begin{equation}\label{eq-1-q12-1-q14}
			\frac{(1-q)^2}{(1-q^{12})(1-q^{14})(q;q^2)_{\infty}}=-q-q^4-q^{13}-q^{16}+\frac{N(q)+q^{27}+q^{30}+q^{39}+q^{42}}{(1-q^{12})(1-q^{14})},
		\end{equation}
		which clearly has non-negative coefficients of $q^n$ for $n\ge 17$.
		
		On the other hand, multiplying $(1-q^{14})$ on both sides of \eqref{eq-1-q12-1-q14}, we have
		\begin{equation}
			\frac{(1-q)^2}{(1-q^{12})(q;q^2)_{\infty}}=(-q-q^4-q^{13}-q^{16})(1-q^{14})+\frac{N(q)+q^{27}+q^{30}+q^{39}+q^{42}}{(1-q^{12})},
		\end{equation}
		which also has non-negative coefficients of $q^n$ for $n\ge 17$. This completes the proof.
\end{proof}

We are now in a position to give a proof of Theorem \ref{thm-main22}.

\noindent{\it{Proof of Theorem \ref{thm-main22}}.} By Lemma \ref{thm-main22-lem1}, it is sufficient to analyze the positivity of the three summands in \eqref{eq-b26-b27q6}.
Clearly, the last summand in \eqref{eq-b26-b27q6} has non-negative coefficients. We first study the positivity of the first summand in \eqref{eq-b26-b27q6}. To be specific, we will show that
\[\frac{q^6C(q)(1-q)^2}{(1-q^{12})(1-q^{14})(q;q^2)_{\infty}}\]
has non-negative coefficients of $q^n$ for $n\ge 55$. In fact, note that $C(q)$ is a polynomial of $q$ with positive coefficients and degree $32$. Moreover  Lemma \ref{lem-main22-1} yields that
\[\frac{(1-q)^2}{(1-q^{12})(1-q^{14})(q;q^2)_{\infty}}\]
has non-negative coefficients of $q^n$ for $n\ge 17$. From the above analysis, we derive the positivity of the first summand.

We then deduce the positivity of the second summand. From Lemma \ref{lem-main22-1} we see that
\[\frac{(1-q)^2}{(1-q^{12})(q;q^2)_{\infty}}\]
has non-negative coefficients of $q^n $ for $n\ge 17$. Moreover, $q^{16}(1+q)$ is a polynomial with degree $17$ and non-negative coefficients. Thus  we deduce that the coefficients of $q^n$ in the second summand in \eqref{eq-b26-b27q6} {are} non-negative for $n\ge 34$.

From the above analysis, we derive that  the coefficients of $q^n$ in $\sum_{n=0}^{\infty}(b_{2,6}(n)-b_{2,7}(n))q^n$ are non-negative for $n\ge 55$. Thus after checking $b_{2,6}(n)\ge b_{2,7}(n)$ when $1\le n\le 54$ and $n\ne 7$, we complete the proof.
\qed

\noindent{\bf Acknowledgments.}   This work was supported by the National Science Foundation of China grants 12371336 and 12171358.

\end{document}